\documentclass[11pt]{amsart}
\usepackage{amssymb}
\usepackage{graphics}
\usepackage{latexsym}
\usepackage{amsmath}
\usepackage{amssymb,amsthm,amsfonts}
\usepackage{amscd}
\usepackage[arrow, matrix, curve]{xy}
\usepackage{syntonly}
\usepackage[small,nohug,heads=vee]{diagrams}
\diagramstyle[labelstyle=\scriptstyle]
\title[On  the Griffiths-Yukawa coupling length ]{On the  Griffiths-Yukawa coupling length of some Calabi-Yau families}
\author[Mao Sheng]{Mao Sheng}
\author[Jinxing Xu]{Jinxing Xu}
\email{msheng@ustc.edu.cn}\email{xujx02@ustc.edu.cn}
\address{School of Mathematical Sciences,
University of Science and Technology of China, Hefei, 230026, China}
\thanks{}
\begin{document}
\theoremstyle{plain}
\newtheorem{thm}{Theorem}[section]
\newtheorem{theorem}[thm]{Theorem}
\newtheorem{lemma}[thm]{Lemma}
\newtheorem{corollary}[thm]{Corollary}
\newtheorem{proposition}[thm]{Proposition}
\newtheorem{addendum}[thm]{Addendum}
\newtheorem{variant}[thm]{Variant}
\theoremstyle{definition}
\newtheorem{construction}[thm]{Construction}
\newtheorem{notations}[thm]{Notations}
\newtheorem{question}[thm]{Question}
\newtheorem{problem}[thm]{Problem}
\newtheorem{remark}[thm]{Remark}
\newtheorem{remarks}[thm]{Remarks}
\newtheorem{definition}[thm]{Definition}
\newtheorem{claim}[thm]{Claim}
\newtheorem{assumption}[thm]{Assumption}
\newtheorem{assumptions}[thm]{Assumptions}
\newtheorem{properties}[thm]{Properties}
\newtheorem{example}[thm]{Example}
\newtheorem{conjecture}[thm]{Conjecture}
\numberwithin{equation}{thm}

\newcommand{\sA}{{\mathcal A}}
\newcommand{\sB}{{\mathcal B}}
\newcommand{\sC}{{\mathcal C}}
\newcommand{\sD}{{\mathcal D}}
\newcommand{\sE}{{\mathcal E}}
\newcommand{\sF}{{\mathcal F}}
\newcommand{\sG}{{\mathcal G}}
\newcommand{\sH}{{\mathcal H}}
\newcommand{\sI}{{\mathcal I}}
\newcommand{\sJ}{{\mathcal J}}
\newcommand{\sK}{{\mathcal K}}
\newcommand{\sL}{{\mathcal L}}
\newcommand{\sM}{{\mathcal M}}
\newcommand{\sN}{{\mathcal N}}
\newcommand{\sO}{{\mathcal O}}
\newcommand{\sP}{{\mathcal P}}
\newcommand{\sQ}{{\mathcal Q}}
\newcommand{\sR}{{\mathcal R}}
\newcommand{\sS}{{\mathcal S}}
\newcommand{\sT}{{\mathcal T}}
\newcommand{\sU}{{\mathcal U}}
\newcommand{\sV}{{\mathcal V}}
\newcommand{\sW}{{\mathcal W}}
\newcommand{\sX}{{\mathcal X}}
\newcommand{\sY}{{\mathcal Y}}
\newcommand{\sZ}{{\mathcal Z}}
\newcommand{\A}{{\mathbb A}}
\newcommand{\B}{{\mathbb B}}
\newcommand{\C}{{\mathbb C}}
\newcommand{\D}{{\mathbb D}}
\newcommand{\E}{{\mathbb E}}
\newcommand{\F}{{\mathbb F}}
\newcommand{\G}{{\mathbb G}}
\newcommand{\HH}{{\mathbb H}}
\newcommand{\I}{{\mathbb I}}
\newcommand{\J}{{\mathbb J}}
\renewcommand{\L}{{\mathbb L}}
\newcommand{\M}{{\mathbb M}}
\newcommand{\N}{{\mathbb N}}
\renewcommand{\P}{{\mathbb P}}
\newcommand{\Q}{{\mathbb Q}}
\newcommand{\R}{{\mathbb R}}
\newcommand{\SSS}{{\mathbb S}}
\newcommand{\T}{{\mathbb T}}
\newcommand{\U}{{\mathbb U}}
\newcommand{\V}{{\mathbb V}}
\newcommand{\W}{{\mathbb W}}
\newcommand{\X}{{\mathbb X}}
\newcommand{\Y}{{\mathbb Y}}
\newcommand{\Z}{{\mathbb Z}}
\newcommand{\id}{{\rm id}}
\newcommand{\rank}{{\rm rank}}
\newcommand{\END}{{\mathbb E}{\rm nd}}
\newcommand{\End}{{\rm End}}
\newcommand{\Hom}{{\rm Hom}}
\newcommand{\Hg}{{\rm Hg}}
\newcommand{\tr}{{\rm tr}}
\newcommand{\Sl}{{\rm Sl}}
\newcommand{\Gl}{{\rm Gl}}
\newcommand{\Cor}{{\rm Cor}}
\newcommand{\Aut}{\mathrm{Aut}}
\newcommand{\Sym}{\mathrm{Sym}}
\newcommand{\ModuliCY}{\mathfrak{M}_{CY}}
\newcommand{\HyperCY}{\mathfrak{H}_{CY}}
\newcommand{\ModuliAR}{\mathfrak{M}_{AR}}
\newcommand{\Modulione}{\mathfrak{M}_{1,n+3}}
\newcommand{\Modulin}{\mathfrak{M}_{n,n+3}}
\newcommand{\Gal}{\mathrm{Gal}}
\newcommand{\Spec}{\mathrm{Spec}}
\newcommand{\Jac}{\mathrm{Jac}}

\newcommand{\fM}{\mathfrak{M}_{AR}}

\newcommand{\proofend}{\hspace*{13cm} $\square$ \\}
\maketitle

\begin{abstract}
We determine the  Griffiths-Yukawa coupling length  of the Calabi-Yau universal  families coming from hyperplane arrangements.
\end{abstract}

\section{Introduction}\label{section:introduction}

In the study of moduli spaces of Calabi-Yau (CY) manifolds, an effective method is to investigate the associated variation of Hodge structure. To every variation of Hodge structure $\V$, we can associate an interesting numerical invariant $\varsigma(\V)$, called the  Griffiths-Yukawa coupling length, which is introduced in \cite{VZ}.  The connection of Griffiths-Yukawa coupling length with Shararevich¡¯s conjecture for CY manifolds has been intensively studied
(see e.g., \cite{LTYZ,Zhang Yi}). It has been shown that, for example, a CY family with maximal Griffiths-Yukawa coupling length is rigid.

On the other hand, among the various Calabi-Yau moduli spaces, the ones coming from hyperplane arrangements are particular interesting, due to their analogue of elliptic curves and their relations to Gross' geometric realization problem (see e.g., \cite{GSSZ,G,SXZ2}).

The main purpose of this note is the determination  of  the Griffiths-Yukawa coupling length  for the CY families coming from hyperplane arrangements. More precisely,  let $m,n,r$ be positive integers satisfying the  condition:
\begin{equation}\label{CY-condition}
  r| m,\ \  n= m-\frac{m}{r}-1.
\end{equation}
Let $\fM$ be the coarse moduli space of  ordered $m$ hyperplane arrangements  in $\P^n$ in general position and $\tilde{\mathcal{X}}_{AR}\xrightarrow{\tilde{f}} \mathfrak{M}_{AR}$ be the family of CY $n$-folds which is obtained by a resolution of $r$-fold covers of $\P^n$ branched along $m$ hyperplanes in general position. This family gives a weight $n$ complex variation of Hodge structure ($\C$-VHS ) $\tilde{\V}_{pr}:=(R^n\tilde{f}_{*}\C)_{pr}$ over $\mathfrak{M}_{AR}$.  Our main result is:
\begin{theorem}\label{thm: main thm}
The  Griffiths-Yukawa coupling length of $\tilde{\V}_{pr}$ is $\varsigma(\tilde{\V}_{pr})=\frac{m}{r}-1$.
\end{theorem}

\begin{remark}
By the theorem above, among the Calabi-Yau families coming from hyperplane arrangements, the ones with minimal Griffiths-Yukawa coupling length (i.e. $\varsigma(\tilde{\V}_{pr})=1$) are those with $(m,n,r)=(n+3,n,\frac{n+3}{2})$, and these are exactly the families considered in \cite{SXZ}. In particular,  we get a positive answer to one of the questions listed at the end of \cite{SXZ}. Besides, the ones with maximal Griffiths-Yukawa coupling length (i.e. $\varsigma(\tilde{\V}_{pr})=n$) are  those with $(m,n,r)=(2n+2,n,2)$, and these are exactly the Calabi-Yau families coming from double covers of $\P^n$ branched along $2n+2$ hyperplanes in general position.
\end{remark}

In section \ref{section:definitions}, we recall the notion of $\C$-VHS, which is more general than the usual notion of $\Q$-VHS or $\R$-VHS. Then we recall the definition of the Griffiths-Yukawa coupling length in the context of $\C$-VHS. In section \ref{section:VHS from hyperplane arrangement}, we introduce the various $\C$-VHS coming from hyperplane arrangements.  The key observation leading to the main result is a special locus $\mathfrak{M}_C$ in $\mathfrak{M}_{AR}$ which parameterizes distinct points on $\P^1$. We use the weight one $\C$-VHS associated to $\mathfrak{M}_C$ to compute some of the Hodge numbers of $\tilde{\V}_{pr}$, hence obtain an upper bound of $\varsigma(\tilde{\V}_{pr})$. In section \ref{section:Jacobian ring computations}, we use the tool of Jacobian rings to compute the Higgs map associated to the weight one $\C$-VHS on $\mathfrak{M}_C$. In this way, we obtain the lower bound of $\varsigma(\tilde{\V}_{pr})$ and finish the proof of the main result.

\section{Definitions}\label{section:definitions}

Throughout this section, we let $M$ be a complex manifold and $\sO_M$ be the sheaf of  holomorphic functions on $M$. We identify holomorphic vector bundles of finite rank on $M$ and sheaves of locally free  $\sO_M$-modules of finite rank on $M$ by the well known way. We first review the definition of complex variation of Hodge structure, which seem to vary slightly according to the source.

\begin{definition} (c.f.  \cite{D})
A complex variation of Hodge structure ($\C-$VHS) of weight $n$ on $M$ is a local system $\V$ of finite dimensional $\C$-vector spaces on $M$ together with a filtration of holomorphic vector bundles:
$$
F^n\sV\subset F^{n-1}\sV\subset \cdots \subset F^1\sV\subset F^0\sV=\sV:=\V\otimes_{\C}\sO_M
$$
satisfying the following Griffiths transversality  condition:
$$
DF^p\sV\subset F^{p-1}\sV\otimes_{\sO_M}\Omega_M, \ \ \forall \ 1\leq p\leq n,
$$
where $D:\sV\rightarrow \sV\otimes_{\sO_M}\Omega_M $ is the Gauss-Manin connection induced by the local system $\V$. The filtration $F^{\cdot}\sV$ is called the Hodge filtration and the integers $h^{p,n-p}(\V):=rank \frac{F^p\sV}{F^{p+1}\sV}$ $(0\leq p\leq n)$ are called the Hodge numbers. Here by convention, $F^{n+1}\sV=0$.
\end{definition}

\begin{remark}
If $\V$ is a weight $n$ rational variation of Hodge strucutre ($\Q$-VHS), then it is easy to see $\V\otimes_{\Q}\C$ is admits a structure of  $\C$-VHS of  weight $n$ .
\end{remark}

Obviously, a morphism $\phi: \W\rightarrow \V$ between two $\C$-VHS of weight $n$ on $M$ is a morphism of local systems preserving the Hodge filtration.

If $(\V, F^n\sV\subset F^{n-1}\sV\subset \cdots \subset F^1\sV\subset F^0\sV=\sV:=\V\otimes_{\C}\sO_M)$ is a $\C$-VHS of weight $n$ on $M$, then by the Griffiths transversality, $\forall \  0\leq p\leq n$, the Gauss-Manin connection $D$ induces a $\sO_M$-linear homomorphism:
$$
\theta^{p,n-p}:E^{p,n-p}\rightarrow E^{p-1,n-p+1}\otimes_{\sO_M}\Omega_M
$$
where $E^{p,n-p}:=\frac{F^p\sV}{F^{p+1}\sV}$. For $x\in M$ and $v\in T_xM$, we denote the contraction of $\theta^{p,n-p}$ with $v$ by
$$
\theta_x^{p,n-p}(v):E_x^{p,n-p}\rightarrow E_x^{p-1,n-p+1},
$$
which is a $\C$-linear map between the fibers $E_x^{p,n-p}$ and $E_x^{p-1,n-p+1}$. We call  $(E:=\oplus_{p=0}^{n}E^{p,n-p}, \theta:=\oplus_{p=0}^n \theta^{p,n-p})$ the Higgs bundle associated to the $\C$-VHS $\V$. A morphism between two $\C$-VHS of weight $n$ induces a morphism between the associated Higgs bundles in an obviously way.

Let $\V$ be a $\C$-VHS of weight $n$ over $M$ and $(E,\theta)$ the associated Higgs bundle. For every $q$ with $1\leq q\leq n$, the $q$th iterated Higgs field
$$
E^{n,0}\xrightarrow{\theta^{n,0}}E^{n-1,1}\otimes \Omega_{M}\xrightarrow{\theta^{n-1,1}}\cdots \xrightarrow{\theta^{n-q+1,q-1}}E^{n-q,q}\otimes S^{q}\Omega_M
$$
defines a morphism
$$
\theta^q: Sym^q(TM)\rightarrow Hom(E^{n,0}, E^{n-q,q})
$$
where $TM$ is the holomorphic tangent bundle of $M$.
The Griffiths-Yukawa coupling length $\varsigma(\V)$ of $\V$ is defined by
$$
\varsigma(\V)=min\{q\geq 1 \mid \theta^q=0\}-1.
$$

\begin{lemma}\label{lemma:embedding VHS}
Let $(\V,F^{\cdot}\sV)$ and $(\tilde{\V},F^{\cdot}\tilde{\sV})$ be $\C$-VHS of weight $n$ on $M$ with Hodge numbers $h^{n,0}(\V)=h^{n,0}(\tilde{\V})=1$. Suppose $\phi: \V \rightarrow \tilde{\V}$ is a morphism of $\C$-VHS satisfying: $\forall \ 0\leq p\leq n$, $\forall \ x\in M$, the induced linear map between the fibers of the  associated Higgs bundles $\phi_x^{p,n-p}: E^{p,n-p}_x\rightarrow \tilde{E}^{p,n-p}_x$ is injective. Then $\varsigma(\V)=\varsigma(\tilde{\V})$.
\end{lemma}

\begin{proof}
The assumptions give us the following commutative diagram, $\forall \ x\in M$, $\forall \ 1\leq q\leq n$ :
\begin{diagram}
E^{n,0}_x & \rTo^{\theta_x^{n,0}}&E_x^{n-1,1}\otimes \Omega_{M,x}&\rTo^{\theta_x^{n-1,1}}&\cdots & \rTo{\theta_x^{n-q+1,q-1}}&E_x^{n-q,q}\otimes S^{q}\Omega_{M,x}\\
\dTo^{\wr}_{\phi_x^{n,0}} &        &\dInto_{\phi_x^{n-1,1}}\otimes id&                           &     &
                          &\dInto_{\phi_x^{n-q,q}}\otimes id \\
\tilde{E}^{n,0}_x & \rTo^{\tilde{\theta}_x^{n,0}}&\tilde{E}_x^{n-1,1}\otimes \Omega_{M,x}&\rTo^{\tilde{\theta}_x^{n-1,1}}&\cdots & \rTo{\tilde{\theta}_x^{n-q+1,q-1}}&\tilde{E}_x^{n-q,q}\otimes S^{q}\Omega_{M,x}
\end{diagram}
where $\phi_x^{n,0}$ is an isomorphism, and $\phi_x^{n-q,q}$ is injective, $\forall \ 1\leq q\leq n$. It is easy to see that $\varsigma(\V)=\varsigma(\tilde{\V})$ follows from this commutative diagram and the definitions of $\varsigma(\V)$ and $\varsigma(\tilde{\V})$.
\end{proof}

If $V_1$, $V_2$ are linear subspaces of a $\C$-linear space $V$, for $0\leq p\leq n$, we let $\wedge^pV_1\wedge^{n-p}V_2$ denote the linear subspace of $\wedge^n V$ spanned  by elements in the set
$$
\{e_1\wedge \cdots \wedge e_p\wedge e_{p+1}\wedge \cdots \wedge e_n | e_1,\cdots, e_p \in V_1, e_{p+1},\cdots, e_n\in V_2\}.
$$
Similarly, if $W_i$ ($1\leq i\leq n$) are linear subspaces of a $\C$-linear space $W$, we let $W_1\wedge \cdots \wedge W_n$ denote the linear subspace of $\wedge^n W$ spanned by elements in the set
$$
\{e_1\wedge \cdots \wedge e_n | e_i \in W_i, 1\leq i\leq n\}.
$$

Now suppose $(\V, F^{\cdot}\sV)$ is a $\C$-VHS of weight $n$ on $M$, for any $m\geq 1$, we endow the local system $\wedge^m \V$ a Hodge filtration $F^{\cdot}\wedge^m \V$ such that $\wedge^m \V$ becomes a $\C$-VHS of weight $mn$ with this Hodge filtration. The Hodge filtration is defined as follows:

$\forall \ 0\leq p\leq mn$, $\forall \ x \in M$, the fiber of the holomorphic bundle $F^p \wedge^m \V$ at $x$ is $$
\sum_{i_1+\cdots +i_n\geq p} F^{i_1}\sV_x\wedge F^{i_2}\sV_x\wedge \cdots \wedge F^{i_n}\sV_x \subset \wedge^n\sV_x.
$$
With this definition of Hodge filtration, the Griffiths transversality is easy to verify.

We give the following lemma for the purpose of later use:

\begin{lemma}\label{lemma:wedge product}
Let $\W$ be a weight one $\C$-VHS on $M$ with associated Higgs bundle $(F=F^{1,0}\oplus F^{0,1}, \eta: F^{1,0}\rightarrow F^{0,1}\otimes \Omega_M)$. The Hodge numbers are $h^{1,0}(\W)=n$, $h^{0,1}(\W)=k-1\leq n$. Suppose there exist $x\in M $ and $v\in T_xM$ satisfying that the Higgs map $\eta_x(v): F^{1,0}_x\rightarrow F^{0,1}_x$ is surjective. Then  the Griffiths-Yukawa coupling length of $\wedge^n \W$ is $\varsigma(\wedge^n\W) =k-1$.
\end{lemma}

\begin{proof}
Let $\V=\wedge^n \W$ be the $\C$-VHS of weight $n$ on $M$. Denote the Higgs bundle associated to $\V$ by $(E=\oplus E^{p,q}, \theta =\oplus \theta^{p,q})$. A direct  computation shows that $h^{n-q,q}(\V)=0$, $\forall \ q\geq k$. From this  we get $\varsigma(\V)\leq k-1$. In order to prove $\varsigma(\V)\geq k-1$, it suffices to show that the iterated Higgs map
$$
\theta^{k-1}_x(v^{k-1}): E^{n,0}_x \rightarrow E^{n-k+1, k-1}_x
$$
is nonzero.

By definitions, for any $0\leq q\leq n$, we can identify $E^{n-q,q}_x$ with $\wedge^{n-q}F^{1,0}_x\wedge^q F^{0,1}_x \subset \wedge^n F_x$. With these identifications, we have the following commutative diagram:

$$
 \begin{array}{ccc}
   E_x^{n,0} & \xrightarrow{\theta_x^{k-1}(v^{k-1})} & E_x^{n-k+1,k-1} \\
   \downarrow{\wr} &  & \downarrow{\wr} \\
   \wedge^nF_x^{1,0} & \xrightarrow{\wedge^n\eta_x(v)} & \wedge^{n-k+1}F_x^{1,0}\wedge^{k-1}F_x^{0,1}
   \end{array}
 $$
where for $e_1\wedge\cdots e_n \in \wedge^nF_x^{1,0}$,
$$
  \wedge^n\eta_x(v)(e_1\wedge\cdots e_n) :=(k-1)!\sum_{1\leq i_1<\cdots <i_{k-1}\leq n} e_1\wedge \cdots \wedge \eta_x(v)e_{i_1}\wedge \cdots \wedge \eta_x(v)e_{i_{k-1}}\wedge \cdots\wedge e_n.
$$

From this diagram and the explicit expression of $\wedge^n\eta_x(v)$, it is easy to deduce the non-vanishing of the map
$\theta^{k-1}_x(v^{k-1})$ from the surjectivity of the map $\eta_x(v): F^{1,0}_x\rightarrow F^{0,1}_x$.

\end{proof}

\section{$\C$-VHS from hyperplane arrangements}\label{section:VHS from hyperplane arrangement}
Now the meaning of letters in the tuple $(m,n,r, \zeta)$ will be fixed to the end of the paper: $m$, $n$, $r$ are  positive integers  satisfying the condition (\ref{CY-condition}), and   $\zeta$ is  a  fixed primitive $r$-th root of unity. If the cyclic group $\Z/r\Z=<\sigma>$ acts on a $\C$-VHS $\V$, we denote $\V_{(i)}$ as the $i$-th eigen-sub $\C$-VHS of $\V$, i.e. the sections of $\V_{(i)}$ consist sections $s$ of $\V$ satisfying $\sigma \cdot s=\zeta^i s$.

An ordered arrangement $\mathfrak{A}=(H_1,\cdots, H_{m})$ of $m$ hyperplanes in $\P^n$ is in general position if no $n+1$ of the hyperplanes intersect in a point.
Let $\mathfrak{M}_{AR}$ denote the coarse moduli space of ordered arrangements of $m$ hyperplanes in $\P^n$ in general position.
As shown in \cite{SXZ2}, $\mathfrak{M}_{AR}$ can be realized as an open subvariety of the affine space $\C^{n(\frac{m}{r}-1)}$ and it admits a natural family $f:\mathcal{X}_{AR}\rightarrow \mathfrak{M}_{AR}$, where each fiber $f^{-1}(\mathfrak{A})$ is the $r$-fold cyclic cover of $\P^n$ branched along the hyperplane arrangement $\mathfrak{A}$. It is easy to see the crepant resolution process in  \cite{SXZ} gives a simultaneous crepant resolution $\pi: \tilde{\mathcal{X}}_{AR}\rightarrow \mathcal{X}_{AR}$ for the family $f$. We denote this smooth projective family of CY manifolds by $\tilde{f}:\tilde{\mathcal{X}}_{AR}\rightarrow \mathfrak{M}_{AR}$.

Let $\mathfrak{M}_{C}$ be the moduli space of ordered distinct $m$ points on $\P^1$ and  $g:\mathcal{C}\rightarrow \mathfrak{M}_{C}$ be the  universal family of $r$-fold cyclic covers of $\P^1$ branched at $m$ distinct points. There is a natural embedding $\mathfrak{M}_{C}\hookrightarrow \mathfrak{M}_{AR}$ (for details, see \cite{SXZ2}, section 2.3).

 We consider the various $\C$-VHS attached to the three families $f$, $\tilde{f}$, $g$:
 $$
 \V:=R^nf_{*}\C, \ \ \tilde{\V}:=R^n\tilde{f}_{*}\C,\ \ \tilde{\V}_{pr}:=(R^n\tilde{f}_{*}\C)_{pr}, \ \ \W:=R^1g_{*}\C,
 $$
where $(R^n\tilde{f}_{*}\C)_{pr}$ means the local system on $\mathfrak{M}_{AR}$ whose fiber over $\mathfrak{A}\in \mathfrak{M}_{AR}$ is the primitive $n$-th cohomology of $\tilde{f}^{-1}(\mathfrak{A})$.
Note that $\V$ is indeed a $\C$-VHS, although the family $f$ is not smooth (see \cite{SXZ2}, section 6). Note also the weights of $\V$, $\tilde{\V}$ and $\tilde{\V}_{pr}$ are $n$, while the weight of $\W$ is one.

Since $\Z/r\Z$ acts naturally on the  families $f:\mathcal{X}_{AR}\rightarrow \mathfrak{M}_{AR}$ and $g:\mathcal{C}\rightarrow \mathfrak{M}_{C}$, we have a decomposition of the  $\C$-VHS into eigen-sub $\C$-VHS:
 $$
 \V=\oplus_{i=0}^{r-1}\V_{(i)},
 \ \ \W=\oplus_{i=0}^{r-1}\W_{(i)}.
 $$

\begin{proposition}\label{prop:basic prop of various VHS} Notations as above, then
\begin{itemize}
  \item[(1)] as $\C$-VHS of weight $n$  on  $\mathfrak{M}_C$, we have $\V_{(1)}\mid_{\mathfrak{M}_{C}}\simeq \wedge^n\W_{(1)}$;
  \item[(2)] the Hodge numbers of $\W_{(1)}$ are: $h^{1,0}(\W_{(1)})=n$, \ \ $h^{0,1}(\W_{(1)})=\frac{m}{r}-1$;

  \item[(3)] the Hodge numbers of $\V_{(1)}$ are:
  $$
  h^{n-q,q}= \left\{
               \begin{array}{ll}
                 {n\choose q}{\frac{m}{r}-1\choose q}, & \hbox{$\ 0\leq q\leq \frac{m}{r}-1$;} \\
                 0, & \hbox{$\frac{m}{r}\leq q\leq n$.}
               \end{array}
             \right.
$$
\item[(4)] $\varsigma(\tilde{\V}_{pr})=\varsigma(\V_{(1)})$;
\item[(5)] $\varsigma(\tilde{\V}_{pr})\leq \frac{m}{r}-1$.

  \end{itemize}

\end{proposition}
\begin{proof}
For $(1)$, one can see \cite{SXZ2}, Proposition 6.3.

$(2)$ follows from a standard computation of the Hodge numbers of cyclic covers of $\P^1$. One can see for example \cite{M}, $(2.7)$.

$(3)$ follows from $(1)$ and $(2)$.

 $(4)$ follows from Lemma \ref{lemma:embedding VHS}. Indeed, one can verify directly that the embeddings $\V_{(1)}\hookrightarrow \V$ and $\tilde{\V}_{pr}\hookrightarrow \tilde{\V}$ satisfy the hypothesis of Lemma \ref{lemma:embedding VHS}, so we get $\varsigma(\V_{(1)})=\varsigma(\V)$, $\varsigma(\tilde{\V}_{pr})=\varsigma(\tilde{\V})$.  One can use  Theorem 5.41 in \cite{PS} to show the natural morphism $\V\rightarrow \tilde{\V}$ satisfies the hypothesis of Lemma \ref{lemma:embedding VHS}, which  gives $\varsigma(\V)=\varsigma(\tilde{\V})$. Combining these equalities, we get $\varsigma(\tilde{\V}_{pr})=\varsigma(\V_{(1)})$.

 $(5)$ follows from $(3)$ and $(4)$.

\end{proof}

\section{Computations in Jacobian ring }\label{section:Jacobian ring computations}

In this section, we keep the  notations in section \ref{section:VHS from hyperplane arrangement}. We want to analyse the Higgs maps associated to the universal family $g: \sC\rightarrow \mathfrak{M}_C$ in some detail. Recall $\mathfrak{M}_C$ is the coarse moduli space of ordered pairwise distinct $m$ points in $\P^1$. It is well known that   $\mathfrak{M}_C$ can be identified with a Zariski open subset $U$ of $\C^{m-3}$ via the map:
\begin{equation}\notag
\begin{split}
U &\xrightarrow{\sim} \mathfrak{M}_C\\
(a_1,\cdots, a_{m-3})&\mapsto ([1:0],[1;1],[0:1], [1:a_1],\cdots, [1:a_{m-3}])
\end{split}
\end{equation}
where $[z_0:z_1]$ are the homogeneous coordinates on $\P^1$. We fix this identification and view $\mathfrak{M}_C$ as a Zariski open subset of $\C^{m-3}$.

For $a=(a_1,\cdots, a_{m-3})\in \mathfrak{M}_C \subset \C^{m-3}$, let $C=g^{-1}(a)$ be the fiber over $a$ of the universal family $\sC$, then $C$ is the $r$-fold cyclic cover of $\P^1$ branched at the $m$ points $[1:0],[1;1],[0:1], [1:a_1],\cdots, [1:a_{m-3}]$.
We let $Y$ denote the  smooth curve  which is the  complete intersection of the $m-2$ hypersurfaces in $\P^{m-1}$ defined by the equations:
\begin{equation}\notag
\begin{split}
&y_{2}^r-(y_0^r+y_1^r)=0;\\
& y_{2+i}^r-(y_0^r+a_{i}y_1^r)=0, \ 1\leq i\leq   m-3.
\end{split}
\end{equation}
Here $[y_0:\cdots:y_{m-1}]$ are the homogeneous coordinates on $\P^{m-1}$. $Y$ is called the Kummer cover of $C$, and when $a$ varies in $\mathfrak{M}_C$, the Kummer covers of $g^{-1}(a)$ form a family of curves over $\mathfrak{M}_C$.

Let $N=\oplus_{j=0}^{m-1}\Z/r\Z$. Consider the  following group
\begin{equation}\notag
\begin{split}
N_1:=Ker(N &\rightarrow \Z/r\Z)\\
(a_j)&\mapsto \sum_{j=0}^{m-1}a_j
\end{split}
\end{equation}
We define a natural  action of $N$ on $Y$.  $\forall \alpha=(a_0,\cdots, a_{m-1})\in N$, the action of $\alpha$ on $Y$ is induced by
\begin{equation}\notag
\alpha\cdot y_j :=\zeta^{a_j}y_j, \ \ \forall \ 0\leq j\leq m-1.
\end{equation}
Recall $\zeta $ is a fixed $r$-th primitive root of unity.

\begin{proposition}\label{prop:relations between C and Y}
The following statements hold:
\begin{itemize}
\item[(1)] The map $\pi_1: Y\rightarrow \P^1$, $[y_0:\cdots:y_{m-1}]\mapsto [y_0^r:y_{1}^r]$ defines a cover of degree $r^{m-1}$.
\item[(2)]$C\simeq Y/N_1$.
\item[(3)] There exists a natural isomorphism of rational Hodge structures $H^1(C,
\Q)\simeq H^1(Y,\Q)^{N_1}$, where $H^1(Y,\Q)^{N_1}$ denotes the subspace of invariants under $N_1$.
\end{itemize}
\end{proposition}
\begin{proof}
$(1)$ can be verified directly.

$(2)$: By $(1)$, one can verify $Y/N_1$ is a  $r$-fold cyclic cover of $\P^1$ branched at the $m$ points $[1:0],[1;1],[0:1], [1:a_1],\cdots, [1:a_{m-3}]$. Then $C\simeq Y/N_1$ follows from the uniqueness of this kind of covers.

$(3)$ follows from $(2)$ directly.
\end{proof}

Recall $\W=R^1g_{*}\C$ is the weight one $\C$-VHS coming from the universal family $g: \sC\rightarrow \mathfrak{M}_C$, and under the natural $\Z/r\Z$-action, $\W_{(1)}$ is the first eigen-sub $\C$-VHS of $\W$. Let $(F=F^{1,0}\oplus F^{0,1}, \eta:F^{1,0}\rightarrow F^{0,1}\otimes \Omega_{\mathfrak{M}_C})$ be the Higgs bundle  associated to $\W_{(1)}$. At the point $a=(a_1,\cdots, a_{m-3})\in \mathfrak{M}_C \subset \C^{m-3}$, we can identify the fiber $F^{1,0}_a$ with $H^{1,0}(C,\C)_{(1)}$, where $H^{1,0}(C,\C)_{(1)}:=\{\alpha\in H^{1,0}(C,\C) | \sigma^{*}\alpha =\zeta \alpha\}$ is the first eigen subspace of $H^{1,0}(C,\C)$ under the natural $\Z/r\Z=<\sigma>$-action. Similarly, we can identify $F^{0,1}_a$ with  $H^{0,1}(C,\C)_{(1)}$. Under these identifications, it is a standard fact that we have the following commutative diagram (see e.g., \cite{Voisin}, Theorem 10.21):
\begin{equation}\label{equation:commut. diag. for Higgs}
\begin{array}{ccc}
   F^{1,0}_a \otimes T_a \mathfrak{M}_C & \xrightarrow{\eta_a} & F^{0,1}_a \\
   \downarrow{\wr} &  & \downarrow{\wr} \\
   H^{1,0}(C,\C)_{(1)}\otimes T_a\mathfrak{M}_C & \xrightarrow{\psi_a} & H^{0,1}(C,\C)_{(1)}
   \end{array}
\end{equation}
where $\psi_a$ means the composition of the Kodaira-Spencer map $ T_a\mathfrak{M}_C\rightarrow H^1(C, TC)$ and the cup product $H^{1,0}(C,\C)_{(1)}\otimes H^1(C, TC)\rightarrow H^{0,1}(C,\C)_{(1)}$.

Since the isomorphism in Proposition \ref{prop:relations between C and Y}, $(3)$ is equivariant with respect to the $\Z/r\Z$-action (note $\Z/r\Z=N/N_1$ acts naturally on $H^1(Y,\Q)^{N_1}$), we can identify the corresponding  eigen subspaces: $H^{1,0}(C,\C)_{(1)}=H^{1,0}(Y,\C)^{N_1}_{(1)}$, $H^{0,1}(C,\C)_{(1)}=H^{0,1}(Y,\C)^{N_1}_{(1)}$. Under these identifications, we have the commutative diagram:
 \begin{equation}\label{equation:commut. diag. C and Y}
\begin{array}{ccc}
   H^{1,0}(C,\C)_{(1)}\otimes T_a\mathfrak{M}_C  & \xrightarrow{\psi_a} & H^{0,1}(C,\C)_{(1)} \\
   \downarrow{\wr} &  & \downarrow{\wr} \\
   H^{1,0}(Y,\C)^{N_1}_{(1)}\otimes T_a\mathfrak{M}_C & \xrightarrow{\tilde{\psi}_a} & H^{0,1}(Y,\C)^{N_1}_{(1)}
   \end{array}
\end{equation}
where $\tilde{\psi}_a$ is defined similarly as $\psi_a$.

In order to represent  $\tilde{\psi}_a$ more explicitly, we use the tool of Jacobian ring. It is constructed as follows.  In the polynomial ring of $2m-2$ variables $\C[\mu_0,\cdots, \mu_{m-3}, y_0,\cdots, y_{m-1}]$, consider the polynomial
$$
F=\mu_0F_0+\cdots +\mu_{m-3}F_{m-3}
$$
where
\begin{equation}\notag
\begin{split}
&F_0:=y_{2}^r-(y_0^r+y_1^r),\\
&F_i:=y_{i+2}^r-(y_0^r+a_{i}y_1^r), \ 1\leq i\leq m-3.
\end{split}
\end{equation}

Let $J=<\frac{\partial F}{\partial \mu_i}, \frac{\partial F}{\partial y_j} \mid 0\leq i\leq m-3, 0\leq j\leq m-1>$ be the ideal of $\C[\mu_0,\cdots, \mu_{m-3}, y_0,\cdots, y_{m-1}]$ generated by the partial derivatives of $F$. Define the Jacobian ring to be
$$
R:=\C[\mu_0,\cdots, \mu_{m-3}, y_0,\cdots, y_{m-1}]/J
$$

There is a natural bigrading on the polynomial ring $\C[\mu_0,\cdots, \mu_{m-3}, y_0,\cdots, y_{m-1}]$, that is: the $(p,q)-$part $\C[\mu_0,\cdots, \mu_{m-3}, y_0,\cdots, y_{m-1}]_{(p,q)}$ is linearly spanned by the monomials $\Pi_{i=0}^{m-3}\mu_i^{\alpha_i}\Pi_{j=0}^{m-1}y_j^{\beta_j}$ with $\sum_{i=0}^{m-3}\alpha_i=p$, $\sum_{j=0}^{m-1}\beta_j=q$. Since the ideal $J$ is a homogeneous ideal, there is a naturally induced bigrading on $R$, written as  $R=\oplus_{p,q\geq 0} R_{(p,q)}$.

The group $N=\oplus_{j=0}^{m-1}\Z/r\Z$ acts on $R$ through $y_0,\cdots, y_{m-1}$. Explicitly, recall $\zeta$ is a fixed primitive $r-$th root of unit, $\forall \alpha=(\alpha_j)\in N$, we define the action of $\alpha$ on $R$  by
\begin{equation}\notag
\begin{split}
\alpha\cdot y_j&:=\zeta^{\alpha_j}y_j, \ \ \forall \ 0\leq j\leq m-1.\\
\alpha\cdot \mu_i&:=\mu_i,\ \ \forall \ 0\leq i\leq m-3.
\end{split}
\end{equation}

It is obviously that the action of $N$ on $R$ preserves the bigrading.
Let $R_{(p,q)}^{N}$ be the $N-$invariant part of $R_{(p,q)}$, then we have the decomposition of the $N-$invariant subring: $R^{N}=\oplus_{p,q\geq 0} R_{(p,q)}^{N}$. Recall $N_1=Ker (\oplus_{j=0}^{m-1}\Z/r\Z\xrightarrow{\sum}\Z/r\Z)$ is the kernel  subgroup of $N=\oplus_{j=0}^{m-1}\Z/r\Z$ under the summation homomorphism.


\begin{proposition}\label{prop:identification of Hodge structures with Jacobian ring}
The following statements hold:
 \begin{itemize}
 \item[(1)] There are  isomorphisms
 $$
  H^{1,0}(Y,\C)^{N_1}_{(1)}\simeq R^{N}_{(0,mr-m-2r)},\ \ \
  H^{0,1}(Y,\C)^{N_1}_{(1)}\simeq R^{N}_{(1,mr-m-r)}.
 $$
 \item[(2)] Let $(x_1,\cdots, x_{m-3})$ be the coordinates on $\mathfrak{M}_C\subset \C^{m-3}$, define a map
\begin{equation}\notag
\begin{split}
T_{a}\mathfrak{M}_C&\xrightarrow{\phi} R_{(1,r)}\\
\frac{\partial}{\partial x_i}&\mapsto -\mu_iy_1^r
\end{split}
\end{equation}
Under this map and the isomorphisms in $(1)$, we have a commutative diagram

\begin{equation}\label{diagram: reduce to Jacobian ring}
\begin{array}{ccc}
 H^{1,0}(Y,\C)^{N_1}_{(1)}\otimes T_a\mathfrak{M}_C & \xrightarrow{\tilde{\psi}_a} & H^{0,1}(Y,\C)^{N_1}_{(1)} \\
 \downarrow{id\otimes \phi}&  & \downarrow{\simeq} \\
  R^{N}_{(0,mr-m-2r)}\otimes R_{(1,r)} &\xrightarrow{} & R^{N}_{(1,mr-m-r)}
\end{array}
\end{equation}
where the lower  horizontal map  is the natural multiplication in $R$.
\end{itemize}
\end{proposition}

\begin{proof}
 $(1)$ follows from \cite{T}, Corollary 2.5 and its proof.

 $(2)$ follows from $(1)$  and \cite{T}, Proposition 2.6.
\end{proof}

For $\beta=\sum_{i=1}^{m-3}\lambda_i\mu_iy_1^r\in R_{(1,r)}$, let $\psi_{\beta}$ be the following multiplication homomorphism
\begin{equation}\notag
\begin{split}
R^{N}_{(0,mr-m-2r)}&\xrightarrow{\psi_{\beta}}R^{N}_{(1,mr-m-r)}\\
\alpha &\mapsto \alpha\cdot \beta
\end{split}
\end{equation}

Now our key computations are included in the following

\begin{proposition}\label{prop:Jacobian ring map is surjective}
For a generic $\lambda=(\lambda_1,\cdots, \lambda_{m-3})\in \C^{m-3}$, the homomorphism $\psi_{\beta}$ is surjective.
\end{proposition}

\begin{proof}
We first analyse the relations in $R$, in order to obtain bases of the $\C$-linear spaces $R^{N}_{(0,mr-m-2r)}$
and $R^{N}_{(1,mr-m-r)}$.

By the definition, the following relations hold in $R$:
\begin{equation}\label{equ:original relations in R}
\begin{split}
\frac{\partial F}{\partial \mu_0}&=y_{2}^r-(y_0^r+y_1^r)=0;\\
\frac{\partial F}{\partial\mu_i}&=y_{i+2}^r-(y_0^r+a_{i}y_1^r)=0, \ \ 1\leq i\leq m-3;\\
-\frac{\partial F}{r\partial y_0}&=y_0^{r-1}(\mu_0+\mu_1+\cdots+\mu_{m-3})=0;\\
-\frac{\partial F}{r\partial y_1}&=y_1^{r-1}(\mu_0+a_{1}\mu_1+\cdots+a_{m-3}\mu_{m-3})=0;\\
-\frac{\partial F}{r\partial y_{i+2}}&=\mu_i y_{i+2}^{r-1}=0, \ \ 0\leq i\leq m-3;
\end{split}
\end{equation}

From these relations,  it is easy to see the $\C$-linear space
$R^{N}_{(0,mr-m-2r)}$ is linearly spanned by elements in the set
$$
\{y_0^{br}y_1^{cr}\mid b, c\in \Z_{\geq 0}, \ b+c=m-\frac{m}{r}-2\}.
$$
By Proposition \ref{prop:basic prop of various VHS}, (2), Proposition \ref{prop:relations between C and Y}, (2) amd Proposition \ref{prop:identification of Hodge structures with Jacobian ring}, (1), we can deduce
$$
dim R^{N}_{(0,mr-m-2r)}= dim H^{1,0}(Y,\C)^{N_1}_{(1)}=dim H^{1,0}(C,\C)_{(1)}=h^{1,0}(\W_{(1)})=n=m-\frac{m}{r}-1.
 $$
This implies  that $\{y_0^{br}y_1^{cr}\mid b, c\in \Z_{\geq 0}, \ b+c=m-\frac{m}{r}-2 \}$ is a $\C-$basis of $R^{N}_{(0,mr-m-2r)}$.

 From the relations (\ref{equ:original relations in R}), we get the following relations in $R$:
 \begin{equation}\label{equation:basic relations in R}
 \begin{split}
   &\sum_{i=0}^{m-3}\mu_i y_0^r=0; \\
   & \mu_0y_1^r+ \sum_{i=1}^{m-3}a_i\mu_i y_1^r=0;\\
   & \mu_0y_0^r + \mu_0y_1^r=0;\\
    &\mu_iy_0^r+a_i\mu_iy_1^r=0, \ \ 1\leq i\leq m-3.
    \end{split}
\end{equation}

 The relations above imply directly that the $\C$-linear space $R^{N}_{(1,mr-m-r)}$ is linearly spanned by elements in the set $\{\mu_iy_1^{mr-m-r}\mid i=1,2,\cdots, m-3.\}$.
 Next we want to get a $\C$-basis of $R^{N}_{(1,mr-m-r)}$ from this set.

By the  relations (\ref{equation:basic relations in R}), we know in $R$,
\begin{equation}\notag
 \begin{split}
   & \mu_0y_0^r + \mu_0y_1^r=0;\\
    &\mu_iy_0^r+a_i\mu_iy_1^r=0, \ \ 1\leq i\leq m-3.
    \end{split}
\end{equation}
Then we get $\forall e\geq 0$,
\begin{equation}\label{equation:1 identities}
\begin{split}
  &\mu_0y_0^{er} =  (-1)^e\mu_0y_1^{er};  \\
  &\mu_iy_0^{er} =  (-1)^e a_i^e\mu_iy_1^{er}, \ \ 1\leq i\leq m-3.
\end{split}
\end{equation}
From  the  relations (\ref{equation:basic relations in R}) again,
we get $\forall e\geq 0$,
\begin{equation}\label{equation: 2 identities}
\begin{array}{ccc}
  \mu_0y_0^{er}+\mu_1y_0^{er}+\cdots +\mu_{m-3}y_0^{er} & = & 0 \\
  \mu_0y_1^{er}+a_1\mu_1y_1^{er}+\cdots +a_{m-3}\mu_{m-3}y_1^{er} & = & 0
\end{array}
\end{equation}

From the identities (\ref{equation:1 identities}) and (\ref{equation: 2 identities}), we get $\forall e\geq 0$,
$$
(a_1^e-a_1)\mu_1y_1^{er}+\cdots +(a_{m-3}^e-a_{m-3})\mu_{m-3}y_1^{er}=0
$$
From this we get the following identity
\begin{equation}\notag
\left(
  \begin{array}{cccc}
    a_1^2-a_1 & a_2^2-a_2 & \cdots & a_{m-3}^2-a_{m-3} \\
   a_1^3-a_1 & a_2^3-a_2 & \cdots & a_{m-3}^3-a_{m-3} \\
    \vdots & \vdots &  & \vdots \\
   a_1^{m-\frac{m}{r}-1}-a_1 & a_2^{m-\frac{m}{r}-1}-a_2 & \cdots & a_{m-3}^{m-\frac{m}{r}-1}-a_{m-3} \\
  \end{array}
\right)\left(
         \begin{array}{c}
           \mu_1y_1^{mr-m-r} \\
           \mu_2y_1^{mr-m-r} \\
           \vdots \\
           \mu_{m-3}y_1^{mr-m-r} \\
         \end{array}
       \right)=0.
\end{equation}
Note that $a\in \mathfrak{M}_C$ implies that $\forall \ 1\leq i\leq m-3$, $a_i\neq 0, 1$ and $\forall \ i\neq j$, $a_i\neq a_j$. So that from the matrix equality above, we know that any $m-3-(m-\frac{m}{r}-2)=\frac{m}{r}-1$ distinct  elements in $\{\mu_iy_1^{mr-m-r}\mid i=1,2,\cdots, m-3.\}$ form a $\C-$basis of $R^{N}_{(1,mr-m-r)}$. We can represent the map $\psi_{\beta}$ as follows:
\begin{equation}\notag
\begin{split}
&\psi_{\beta}\left(
              \begin{array}{c}
                 y_1^{mr-m-2r} \\
                y_1^{mr-m-3r}y_0^r \\
                \cdots \\
                y_0^{mr-m-2r} \\
              \end{array}
            \right) =\\
            &\left(
                      \begin{array}{cccc}
                        \lambda_1 & \lambda_2 & \cdots & \lambda_{m-3} \\
                       -a_1\lambda_1 &  -a_2\lambda_2 & \cdots & -a_{m-3}\lambda_{m-3} \\
                        \vdots & \vdots & & \vdots \\
                       (-a_1)^{m-\frac{m}{r}-2}\lambda_1 & (-a_2)^{m-\frac{m}{r}-2}\lambda_2 & \cdots & (-a_{m-3})^{m-\frac{m}{r}-2}\lambda_{m-3} \\
                      \end{array}
                    \right)\left(
                             \begin{array}{c}
                               \mu_1y_1^{mr-m-r} \\
                               \mu_2y_1^{mr-m-r} \\
                               \vdots \\
                               \mu_{m-3}y_1^{mr-m-r} \\
                             \end{array}
                           \right)
                           \end{split}
\end{equation}
From this matrix representation, we can easily see that for a generic $\lambda=(\lambda_1,\cdots, \lambda_{m-3})\in \C^{m-3}$, the homomorphism $\psi_{\beta}$ is surjective.
\end{proof}

Now we can prove our main theorem. Recall the notations from section \ref{section:VHS from hyperplane arrangement}. We have:
\begin{theorem}\label{thm:main theorem}
$\varsigma(\tilde{\V}_{pr})=\frac{m}{r}-1$.
\end{theorem}
\begin{proof}
In Proposition \ref{prop:basic prop of various VHS}, $(5)$, we have established the inequality $\varsigma(\tilde{\V}_{pr})\leq \frac{m}{r}-1$. So it suffices to prove $\varsigma(\tilde{\V}_{pr})\geq \frac{m}{r}-1$.

Combining the commutative diagrams   (\ref{equation:commut. diag. for Higgs}), (\ref{equation:commut. diag. C and Y}), (\ref{diagram: reduce to Jacobian ring}) and Proposition \ref{prop:Jacobian ring map is surjective} together shows that for any $a\in \mathfrak{M}_C$, for a generic tangent vector $v\in T_a\mathfrak{M}_C$,  the Higgs map $\eta_a(v): F_a^{1,0}\rightarrow F_a^{0,1}$ associated to the $\C$-VHS $\W_{(1)}$ is surjective.
This and Proposition \ref{prop:basic prop of various VHS}, $(2)$ imply $\W_{(1)}$ satisfies the hypothesis of Lemma \ref{lemma:wedge product}. Applying  this lemma to $\W_{(1)}$, we get $\varsigma(\wedge^n \W_{(1)})=\frac{m}{r}-1$. This gives $\varsigma(\V_{(1)}|_{\mathfrak{M}_C})=\frac{m}{r}-1$ by Proposition \ref{prop:basic prop of various VHS}, $(1)$. By the definition of the length of Griffith-Yukawa coupling, $\varsigma(\V_{(1)})\geq \varsigma(\V_{(1)}|_{\mathfrak{M}_C})$. So finally we get  the desired inequality $\varsigma(\tilde{\V}_{pr})\geq  \frac{m}{r}-1$ by combining the (in)equalities above and Proposition \ref{prop:basic prop of various VHS}, $(4)$.

\end{proof}

\textbf{Acknowledgements}
This work is supported by Wu Wen-Tsun Key Laboratory of Mathematics, Chinese Academy of Sciences. The first named author is supported by National Natural Science Foundation of China (Grant No. 11622109, No. 11721101).

\end{document}